\documentclass[10pt,a4paper,reqno]{amsart}
\usepackage{amsfonts,amsthm,latexsym,amsmath,amssymb,amscd,amsmath, epsf}
\usepackage{graphicx, epigraph}
\usepackage[all,cmtip]{xy}

\usepackage{graphicx}
\newtheorem{theorem}{Theorem}
\newtheorem{proposition}[theorem]{Proposition}
\newtheorem{lemma}[theorem]{Lemma}
\newtheorem{definition}{Definition}
\newtheorem{corollary}{Corollary}

\newcommand{\bC}{\mathbb{C}}
\newcommand{\be}{\begin{equation}}
\newcommand{\ee}{\end{equation}}

\newcommand {\Ga}{\Gamma}

\newcommand{\HH}{\mathcal{H}}

\newcommand{\C}{\mathcal{C}}
\newcommand{\CP}{\mathbb{CP}}

\newcommand {\df}{\text{pdef}}
\newcommand{\OO}{\mathcal{O}}

\begin{document}
          \numberwithin{equation}{section}

          \title[A note on planarity stratification of Hurwitz spaces]
          {A note on planarity stratification of Hurwitz spaces}

\author[J. Ongaro]{Jared Ongaro}
\address{    Department of Mathematics,
            Stockholm University,
            S-10691, Stockholm, Sweden}
\email{ongaro@math.su.se}

\author[B. Shapiro]{Boris Shapiro}
\address{    Department of Mathematics,           Stockholm University,           S-10691, Stockholm, Sweden}
\email{shapiro@math.su.se}

\begin{abstract} One can easily show that any meromorphic function on a complex closed Riemann surface  can be represented as a composition of a birational map of this surface to $\CP^2$  and  a  projection of the image curve  from an appropriate point $p\in \CP^2$ to the pencil of lines through $p$. We introduce a natural stratification of  Hurwitz spaces according to the minimal degree of a plane curve such that a given meromorphic function can be represented in the above way and calculate the dimensions of these strata. We observe that they     are closely related to a family of Severi varieties studied earlier by J.~Harris, Z.~Ran and I.~Tyomkin.   
\end{abstract}
\maketitle

\section{Basic definitions and facts}  
In what follows we will always work over  the field $\bC$ of complex numbers and by a genus $p_g(\C)$ of a (singular) curve $\C$ we mean its geometric genus, i.e. the genus of its normalization. We start with the following statement. 

\begin{proposition}\label{pr:birational}
Any meromorphic function $f:\C\to \CP^1$ on a complex closed Riemann surface $\C$ can be represented as $f=\pi_p\circ \nu$ where   $\nu: \C \to \CP^2$ is a birational mapping of $\C$ to its image and  $\pi_p: \nu(\C)\to \CP^1$ is the  projection of the image curve $\nu(\C)$ from a point $p\in \CP^2$ to the pencil of lines through $p$. 
\end{proposition}

\begin{proof} Let $\bC(\C)$ be the field of meromorphic functions on $\C$.  Consider its subfield $\bC(f)\subset \bC(\C)$ generated by $f$ (i.e. the set of all rational functions of  $f$). Since $\C$ is one-dimensional the field extension 
$\bC(\C):\bC(f)$ is finite.  Choose any meromorphic  function $g: \C\to \CP^1$ generating this extension. Removing a point from $\CP^1$ and its inverse images under $f$ and $g$, we get a birational mapping $\C\setminus \text{ \{finite set\} }  \to \bC^2$ given by the pair $(f,g)$.   
Its compactification gives a birational mapping $\nu:\C\to \CP^2$. Projection "along the second coordinate" gives a presentation  of the original  meromorphic function $f:\C\to \CP^1$   as $f=\pi_p \circ \nu$. 
\end{proof}

Obviously if $\nu$ maps  $\C$ birationally on its image and   $f= \pi_p \circ \nu$ for some point $p\in \CP^2$, then $\deg(\nu(\C))=\deg f$ if and only if $p\notin \nu(\C)$ and $\deg(\nu(C))>\deg f$ if $p\in \nu(\C)$.

\begin{definition}\label{defect}{}
The {\bf planarity defect} $\df(f)$ of a meromorphic function $f:\C\to \CP^1$ equals
$$\df(f):=\min_\nu (\deg(\nu(\C))-\deg (f)$$
such that $f= \pi_p\circ \nu $  as above. 
\end{definition}

We start with the following simple observation. 

\begin{lemma}\label{defectriv}   Given $f: \C\to \CP^1$, then $\df(f)=0$ if and only if $h^0(f^\star(\OO_1))\ge 3$, and  for  almost any  point $p\in \C $ and any other point  $q\neq p$, 
$$h^0(f^\star(\OO(1))-p-q)=h^0(f^\star(\OO(1)))-2.$$
\end{lemma}

\begin{proof} Indeed, observe that $f$ determines a $1$-dimensional linear subsystem in the complete linear system $f^\star(\OO(1))$. (We count dimensions of linear systems projectively.) Moreover, if $r_f=h^0(f^\star(\OO(1)))\ge 3$ then the system $f^\star(\OO(1))$  defines a map $\phi_f: \C\to \CP^{r_f-1}$  with $r_f-1\ge 2$. If additionally, sections of $f^\star(\OO(1))$ separate each generic point on $\C$ from all other points then $\phi_f$ is birational on the image. The latter condition is   made  explicit  above. Choosing an appropriate $3$-dimensional subsystem of $f^\star(\OO(1))$ including $f$, we get the required statement. 
\end{proof} 

Unfortunately, the second condition is not easy to check in concrete situations, see Remark below.  We say that a linear system $\mathcal L$ on a curve $\C$ is birationally very ample if the image of $\C$ in the projectivized space of its sections is birationally equivalent to $\C$, comp. \cite {Ohb}.

The following sufficient condition of the birational very ampleness of $f^\star(\OO(1))$  is valid.  

\begin{lemma}\label{defec}   If $f: \C\to \CP^1$ has at most one complicated branching point, then $\df(f)=0$ if and only if $h^0(f^\star(\OO(1)))\ge 3$. In particular, under the above assumptions, if   $\deg (f)=d\ge g+2$ where $g$ is the genus of  $\C$  then $\df(f)=0$.
\end{lemma}

\begin{proof} As in Lemma~\ref{defectriv}, the necessary condition for $pdef(f)=0$ is $r_f:=h^0(f^\star(\OO(1)))\ge 3$. 
By Riemann-Roch's formula
\begin{equation}\label{eq:RR}
r_f:=h^0(f^\star(\OO(1)))=d-g+1+h^0(K\setminus (f)_\infty),
\end{equation} 
where $(f)_\infty$ is the pole divisor of $f$.   
 The linear system $f^\star(\OO(1))$ determines the mapping $\phi_f: \C\to \CP^{r_f-1}$. Moreover if $r_f\ge 3$ and $f$ has at most one complicated branching point, then $\phi_f$  defines a birational mapping of $\C$ on its image $\phi_f(\C)$.  Indeed, since $r_f\ge 3$ the only thing that we have to exclude is that $\phi_f: \C\to \phi_f(\C)$ is a non-trivial covering. Assume that $\phi_f: \C\to \phi_f(\C)$ is a non-trivial covering. 
Notice that independently of the fact whether  $\phi_f$ is birational on the image or not, $f=\pi_p\circ \phi_f$ where $\pi_p$ is a projection of $\CP^2\setminus p \to \CP^1$ from some point $p\in \CP^2$.  Also the map $f$ can be lifted in the standard way to $f=\tilde \pi_p \circ \tilde \phi_f$ where  $\tilde \phi_f: \C \to \tilde \phi_f(\C)$ is the standard lift of $\phi_f$ to the normalization $\tilde \phi_f(\C)$ of the image $\phi_f(\C)$, and $\tilde \pi_p$ is the composition of the standard map from the normalization $\tilde \phi_f(\C)$ to the image curve $\phi_f(\C)$ with the projection $\pi_p$. Branching points of $f$ are either the images under $\tilde \pi_p$ of the branching points of $\tilde \phi_f$ or the branching points of $\tilde \pi_p$ itself.  But each branching point of $\tilde \pi_p$ is a non-simple branching point of $f$. Contradiction. The case when $\phi_f(\C)$ is a line in $\CP^2$ is obviously impossible due to the dimension of the linear system $f^\star(\OO(1))$. Finally observe that if $d\ge g+2$ then $r_f$ is at least $3$ by Riemann-Roch's formula \eqref{eq:RR}.\end{proof}

\medskip \noindent
{\bf Remark.}  Observe that for $d\ge g+1,$ any curve $\C$ of genus $g$ admits a meromorphic function of degree $d$ with all simple branching points, i.e. the natural map  $\HH_{g,d}\to \mathcal M_g$ where $\mathcal M_g$ is the moduli space of curves of genus $g$ is surjective, see \cite{Se}. Also for $d\ge 2g+1,$ no genericity assumptions whatsoever on $f$ are required for birational ampleness since $f^\star(\OO(1))$ becomes very ample and  defines an embedding $\C\to \CP^{r_f-1}$. However in the interval $g+2\le d\le 2g$ this linear system might define a non-trivial covering on the image as shown by the next classical example, see Proposition 5.3, \cite {Ha}. This circumstance shows that one needs some additional assumption on the branching points to avoid such coverings. 

\medskip \noindent
{\bf Example.} Let $\C$ be a hyperelliptic curve of genus $g>2$ and let $|L|:\C\to \CP^1$ be the
hyperelliptic map. Let $s_0$ and $s_1$ be a basis for $H^0(L)$. Riemann-Roch's formula  gives that
$h^0(gL) = g + 1 < 2g$. Note that there are precisely $\binom{d+n-1}{n-1}$ monomials of degree 
$d$ in $n$ variables.  Therefore there are precisely $d + 1$ monomials of degree $d$ in $s_0$ 
and $s_1$.
The map $|L|:\C\to \CP^1$ is given by 
$$\C \ni p \mapsto [s_0(p):s_1(p)]\in \CP^1,$$  
while the map $|mL|:\C\to \CP^g$ is given by 
$$P\mapsto [s_0(p)^g : s_0(p)^{g-1}s_1(p) : \dots : s_1(p)^g].$$
But it is now clear that $|mL|:\C\to\CP^g$ can be factored as $|L|:\C\to \CP^1$ followed by the Veronese embedding $V : \CP^1 \to \CP^g$. Hence, the image of $\C$ under the map $|mL|$ is a rational normal curve.
Now suppose that $m > g$. Then Riemann-Roch's formula  gives $h^0(mL) = 2m − g + 1 > m+1$. Thus, $s_0$ and $s_1$ only generate a subspace of $H^0(mL)$ and the above argument no longer works (which is good since $|mL|$ determines  a closed embedding).

\medskip 
We now characterize the vanishing of the planarity defect in different terms. Consider the push-forward sheaf $f_\star{(\OO_\C)}$ on $\CP^1$. Since $f$ is a finite map of compact curves,   $f_\star{(\OO_\C)}$ is a vectorbundle on $\CP^1$ whose dimension equals  $\deg(f)$. By the well-known result of Grothendieck,  $f_\star{(\OO_\C)}=\OO\oplus \sum_{i}\OO(a_i)$, where $a_i$ are integers see e.g. \cite{HaMa}. Observe that all $a_i$ must be negative since $h^0(\OO_\C)=h^0(f_\star{(\OO_\C)})=1$.  

\begin{proposition}\label{pr:tyomkin} For any meromorphic function $f:\C\to \CP^1$ with at most one complicated branching point, its planarity defect $\df(f)$ vanishes if and only if $a_{max}=-1$, where  $a_{max}$ is the maximal of all $a_i$'s in the above notation. 
\end{proposition}

\begin{proof}
Let us  show that under our assumptions $\df(f)=0\leftrightarrow a_{max}=-1$.  We  need to check that $h^0(f^\star(\OO(1))\ge 3$ if and only if $a_{max}=-1$.  Consider $f_\star(f^\star(\OO(1))$. Observe that,  $h^0(f_\star(f^\star(\OO(1)))=
h^0(f^\star(\OO(1)))$ since $f$ is a finite map of compact algebraic curves. Now by projection formula, see Ex. 8.3 in \cite{Ha} 
$$f_\star(f^\star(\OO(1))=\OO(1)\otimes f_\star (\OO_\C)=\OO(1)\oplus \sum_{i}\OO(a_i+1) .$$ 
Since $a_{max}=-1$ then at least one of the terms $\OO(a_i+1)$ equals $\OO$. Therefore $h^0(f_\star(f^\star(\OO(1)))=h^0(\OO(1))+
 \sum_{i}h^0(\OO(a_i+1))\ge 2 +1$. In fact,  $h^0(f_\star(f^\star(\OO(1)))=2+\text{the number of indices  } i \text{  such that } a_i=-1$. 
 \end{proof}




Proposition~\ref{pr:tyomkin}  shows that there is a connection of the  planarity defect with the slope invariants of meromorphic functions and with the  Maroni strata, comp. \cite {DePa} and \cite {Pa}. In fact, the following statement is true.

\begin{proposition}\label{lm:triv}
Given a meromorphic function $f:\C\to \CP^1$ of degree $d$, its planarity defect $pdef(f)$ equals $d^\prime-d$ where $d^\prime$ is the minimal degree of a linear system $\mathcal L$ such that a) $\mathcal L$ is birationally very ample and b) the (effective) divisor of $f^\star(\OO(1))$ is contained in the (effective) divisor of $\mathcal L$. 
\end{proposition} 

\begin{proof}
If $f^\star(\OO(1))$ can serve as $\mathcal L$ then there is nothing to prove.  Otherwise the divisor of $\mathcal L$ must be strictly larger than that of $f^\star(\OO(1))$. In the latter case one can choose a $1$-dimensional linear subsystem of $\mathcal L$ defining a meromorphic function $g: \C\to \CP^1$ which is not proportional to $f$. Consider the map $\psi: \C\to \bC^2$ given by $(f,g)$  and extending it to the map $\tilde \psi:  \C\to \CP^2$ we get the required planarity defect. 
\end{proof}

\subsection {Planarity stratification of small Hurwitz spaces}{}
 
 The small Hurwitz space of degree $d$ functions of genus $g$ curves is defined as: 
$$\HH_{g,d}=\{f: \C\to \CP^1\vert f\; \text{has only simple branched points},  \;\deg f=d\ge 2,\; \text{gen}(\C)=g\ge 0\}.$$
 Recall that $\dim \HH_{g,d}$ equals the number of (simple) branching points of a function from $\HH_{g,d}$ and is given by the formula
$$\dim \HH_{g,d}=2d+2g-2.$$
Small Hurwitz spaces were introduced and studied in substantial details by Clebsch \cite{Cle} and Hurwitz \cite{Hu} at the end of the 19-th century as a tool of investigation of the moduli space $\mathcal M_g$ of genus $g$ curves. 

Proposition~\ref{pr:birational} allows us to introduce the {\it planarity stratification} of $\HH_{g,d}$:
\begin{equation}\label{filtr}
 \HH_{g,d}^{m(g,d)}\subset \HH_{g,d}^{m(g,d)+1} \subset \dots \subset \HH_{g,d}^{M(g,d)}=\HH_{g,d},
\end{equation} 
 where $\HH^l_{g,d}$ consists of all meromorphic functions in $\HH_{g,d}$ whose planarity defect does not exceed $l$. 

We present  some information about this stratification. 

\begin{proposition}\label{pr:filt} For any pair $(g,d)$ where $g\ge 0$ and $d\ge 2$, 
 
\begin{equation}\label{eq:mgd}
m(g,d)=\min_{l\ge 0}\binom{d+l-1}{2}-\binom{l}{2}\ge g.
\end{equation}
which gives 
\begin{equation}\label{eq:mgd1}
m(g,d)=\max{\left(0,\left\lceil \frac{g-\binom{d-1}{2}}{d-1}\right\rceil\right)}.
\end{equation}

\end{proposition}

Moreover the following result holds.  

\begin{theorem}\label{th:dim} In the above notation, given $g,d$  and $l\ge m(g,d)$, the stratum $\HH_{g,d}^l$ is irreducible and its dimension is given by: 
\begin{equation}\label{eq:dimen}
\dim \HH_{g,d}^l=\min{(3d+g+2l-4, 2d+2g-2)}. 
\end{equation}
\end{theorem}

The substantial part of the proof of Theorem~\ref{th:dim} consists of  the following  generalization of the famous result by J.~Harris \cite{Ha} showing that the space of plane curves of genus $g$ and degree $d$ where $g\le \binom{d-1}{2}$ is an irreducible variety whose dense subset consists of nodal curves of genus $g$ (irreducibility of the Severi varieties).  Fixing as above a point $p\in \CP^2$, denote by $S_{g,d,l}$ the variety of reduced irreducible plane curves of degree $d$ having  genus $g$ and  order  $l$ at $p$, where $g\le \binom{d+l-1}{2}-\binom {l}{2}$.  (The order of a plane curve at a given point is the multiplciity of its local intersection at $p$ with  a generic line passing through $p$.) Denote by $W_{g,d,l}\subset S_{g,d,l}$ its subset consisting of curves having an ordinary singularity of order $l$ at $p$ (i.e. transversal intersection of $l$ smooth local branches) and only usual nodes outside $p$.

\begin{theorem}\label{th:Harris} \rm{(1)}  $W_{g,d,l}$ is a smooth manifold of dimensional $3d+g+2l -1$;

 \rm{(2)}  $W_{g,d,l}$ is dense in $S_{g,d,l}$; 

 \rm{(3)}  $S_{g,d,l}$ is irreducible.

\end{theorem} 

The main result of \cite{Ha} is the proof of the same statement in the basic case $l=0$.  Theorem~\ref{th:Harris} follows from already known results of Z.~Ran \cite{Ra1}  and I.~Tyomkin \cite {Tyo}. We first prove  Proposition~\ref{pr:filt} and Theorem~\ref{th:Harris} and then Theorem~\ref{th:dim}. 

\begin{lemma}\label{lm:genus}
The genus of a plane curve decreases by at least $\binom {l}{2}$ by a singularity of order $l$. Moreover the ordinary singularity of order $l$ decreases the genus by exactly $\binom {l}{2}$. 
\end{lemma} 

\begin{proof}
The following algorithm describes by which number the genus of a plane curve of degree $d$  is decreased due to a singularity of order $l$. 

Step 1. Subtract  $\binom {l}{2}$ from $\binom {d-1}{2}$.

Step 2. Blow up the singularity in the plane. The strict transform  of the curve will intersect the exceptional divisor  at $l$ points (counting multiplicities). If each of these (geometrically distinct) points is smooth on the strict transform then the genus drops by exactly $\binom {l}{2}$. 

Step 3. If among the latter points there exist singular we have to repeat the previous step, i.e. if the order of singularity is $s$ then we decrease the genus by $\binom {s}{2}$, then we blow up this point etc.  

After finitely many such steps the curve becomes smooth. (Further blow-ups will not change the genus).  
Thus the minimal decrease of genus equals   $\binom {l}{2}$.
\end{proof}

\begin{proof}[Proof of Proposition~\ref{pr:filt}]  The necessity of \eqref{eq:mgd} is obvious. Indeed we need to construct a plane curve of degree $d+l$ such that it has a singularity of order $l$ at $p$ (so that its projection from $p$ will be a covering of degree $d$) and has a genus of normalization equal to $g$. Having a singularity of order $l$ at $p$ decreases the genus by  at least  $\binom {l}{2}$ compared to $\binom {d+l-1}{2}$ which is the genus of a  smooth curve of degree $d+l$, see Lemma~\ref{lm:genus}  above. Thus the inequality \eqref{eq:mgd} must be satisfied. To show that the least value of $l$ satisfying \eqref{eq:mgd} is enough consider first a configuration of $l$ generic lines through $p$ and additionally $d$ lines in $\CP^2$ in general position. This curve has genus $0$. A slight deformation of this curve by a polynomial vanishing up to order $l+1$ at $p$ will resolve all nodes outside $p$ and given $g=  \min_{l\ge 0}\binom{d+l-1}{2}-\binom{l}{2}$. A more careful deformation will resolve any number of nodes between 
$0$ and $\binom {d}{2}$, see the proof of Theorem~\ref{th:Harris} below. 
The classical  case $g\le \binom {d-1}{2}$ is well presented in  \cite{HaMo}, Appendix E and the general case in \cite{Ra1}.
\end{proof} 

We will need some information about the Hirzebruch surfaces and the Severi varieties on them. For a given non-negative integer $n$,  let $\Sigma_n=Proj(\OO_{\CP^1}\oplus \OO_{\CP^1}(n))$ be the Hirzebruch surface and let $\kappa: \Sigma_n \to \CP^1$ be the natural projection. Consider two non-zero sections $(1,0), (0,\sigma)\in H^0(\CP^1, \OO_{\CP^1}\oplus \OO_{\CP^1}(n))$. They define the maps 
$$\CP^1\setminus \mathbb Z(\sigma)\to \Sigma_n,$$
where $ \mathbb Z(\sigma)$ is the zero locus of $\sigma$. We denote the closures of the images of these maps by $L_0$ and $L_\infty$, respectively. (It is clear that the homological class of $L_\infty$ is independent of the choice of $\sigma$.) The following facts are standard. 

\begin{proposition}\label{pr:hirz}
{\rm(i)} 
 The Picard group $Pic(\Sigma_n)$ is a free abelian group generated by the classes $F$ and $L_\infty$, where $F$ denotes the fiber of projection $\kappa$.  (Observe  that $L_0=nF+L_\infty$.) 

\noindent 
 {\rm(ii)}  The intersection form on $Pic(\Sigma_n)$ is given by $F^2=0, L_\infty^2=-n$, and $F\cdot L_\infty=1$. 
 
 \noindent 
 {\rm(iii)} Any effective divisor $M\in Div(\Sigma_n)$ is linearly equivalent to a linear combination of $F$ and $L_\infty$ with non-negative coefficients.  Moreover, if $M$ does not contain $L_\infty$, then it is linearly equivalent to a combination of $F$ and $L_0$ with non-negative coefficients. 
 
 \noindent 
 {\rm(iv)} The canonical class is given by:
 $$K_{\Sigma_n}=-(2L_\infty+(2+n)F)=-(L_0+L_\infty+2F).$$
 
  \noindent 
 {\rm(v)} Any smooth curve $\C$ with the class $dL_0+kF$ has genus $g(\C)=\frac{(d-1)(dn+2k-2)}{2}$. 

\end{proposition}

Let $g,d,k$ be non-negative integers. We define the Severi variety $V_{g,d,k}\subseteq |\OO_{\Sigma_n}(sL_0+kF)|$ to be the closure of the locus of reduced nodal curves of genus $g$ which do not contain $L_\infty$, and we define $V_{g,d,k}^{irr}\subset V_{g,d,k}$ to be the union of the irreducible components whose generic points correspond to irreducible curves. 

\medskip 
The main result of \cite {Tyo} (see Theorem 3.1 there) is as follows. 

\begin{theorem}\label{th:TyoMain}
For any triple $g,k,d$ of non-negative integers, the variety $V_{g,d,k}^{irr}\subset V_{g,d,k}$ (if non-empty) is irreducible and of expected dimension. 

\end{theorem}

\begin{proof}[Proof of Theorem~\ref{th:Harris}] Let us first naively count the expected dimension of   $S_{g,d,l}$. Indeed, the dimension of the space $S_{g,d,l}$ of plane curves of degree $d+l$ with a singularity at  $p$ of order $l$ equals $\frac{(d+l)(d+l+3)}{2}-\binom{l+1}{2}$. The number of nodes on such a curve under the assumptions that it has genus $g$ equals 
\begin{equation}\label{eq:nodes}
\sharp_{nodes}=\binom{d+l-1}{2}-\binom{l}{2}-g.
\end{equation}
Assuming that each node decreases the dimension by $1$ we get 
$${\rm {exp}}\dim S_{g,d,l}=3d+g+2l-1.$$
 
 We finish our proof with a reference to Theorem~\ref{th:TyoMain}. Indeed, if one blows up the point $p\in \CP^2$ then one gets the first Hirzebruch surface $\Sigma_1$. Observe that plane curves of degree $d+l$ having a singularity of order $l$ at $p$ will after the blow-up lie in the class $(d+l)L_0-lL_\infty=dL_0+lF$. Therefore  the above set $W_{g,d,l}$ of irreducible plane curves having  the singularity of order    $l$ at the point $p$ after this blow-up will transform into the space $V^{irr}_{g, d, l}$ in the above notation. (We consider only the strict transform of each curve disregarding the exceptional divisor.) Thus by the latter Theorem, the variety  $S_{g,d,l}$ is irreducible and of expected dimension. Another proof of essentially the same result directly in the plane $\CP^2$ can be found in \cite {Ra1}, see Irreducibility Theorem on p.~122. 
\end{proof}

\begin{proof}[Proof of Theorem~\ref{th:dim}]  To settle Theorem~\ref{th:dim} we need to prove an analog of Proposition~\ref{pr:smooth} or a weaker statement that such curves equivalent as coverings do not appear in families of $G_3$-orbits. If this is true then $\dim\HH_{g,d}^l=\dim S_{d,l,g}-3$. We need the following Proposition.

Let $S(d,l,g)$ be the Severi variety of all plane curves of degree $d+l$, genus $g$ and ordinary singularity of order $l$ at point p. Let $H(g,d)$ be the Hurwitz space of all branched coverings of degree $d$ and genus $g$. Let $B: S(d,l,g) \to H(g,d)$ be the branching morphism sending each plane curve from $S(d,l,g)$ to the branched covering from its normalization to $\CP^1$ obtained by projection from the point $p$.  

\begin{proposition} \label{pr:final} The dimension of the  fiber of the above map at the curve $N$ obtained by normalization of a generic  curve $\C$ from $S(d,l,g)$ equals $h^0(N,\OO_N(E)) $� 
 where $E$ is the divisor of degree $d+2l$ on $N$ obtained as the pull-back of projection point $p$ together with the pull-back of the general line section of $\C$. 
 (For an arbitrary curve $\C\in S(d,l,g)$ the dimension of the fiber is at most $h^0(N,\OO_N(E)) $.) 
\end{proposition}

\begin{proof} Let $\pi: \Sigma_1\to \CP^2$ be the standard projection of the first Hirzebruch surface $\Sigma_1$ obtained by the blow-up of the point $p$ to $\CP^2$. We have natural maps

\[
\xymatrix{
  N  \ar[r]^{g} \ar[dr]^{f} &\Sigma_1\ar[d] \\
   &\CP^1
}
\]

and exact sequences

\[\label{exact}
\xymatrix{
 0  \ar[r] & T_N \ar[r] \ar@{=}[d] & g^\star T_{\Sigma_1}\ar[r]\ar[d]&N_g\ar[r]\ar[d]&0\\
 0  \ar[r] & T_N \ar[r]                  & f^\star T_{\CP^1}\ar[r]&N_f\ar[r]&0.
}
\]

It is known that $Def^1(N,g)=H^0(N,N_g)$ and $Def^1(N,f)=H^0(N,N_f)$ are the tangent spaces to the space of deformations of the  pairs $(N,g)$ and $(N,f)$ resp. The first one is the tangent space to the Severi variety if $g$ is an immersion; the second one is the tangent space to the Hurwitz space. The sequence \eqref{exact} implies that the kernels $\alpha: g^\star T_{\Sigma_1}\to f^\star T_{\CP^1}$ and $N_g\to N_f$ coincide since $g^\star T_{\Sigma_1} \twoheadrightarrow f^\star T_{\CP^1}$. Since the $\CP^1$-bundle $\Sigma_1\to \CP^1$ admits two non-intersecting sections (the line $L$ and the inverse image of $p$ in $\Sigma_1$) then $Ker\; \alpha=g^\star\OO_{\Sigma_1}(L+\pi^{-1}(p))$.



\end{proof}

For small number of nodes compared to the degree of the irreducible plane curve Theorem~\ref{th:dim} is immediate from the following fact, see Exercise 20 (iii) of \S~1, Appendix A, Ch. 1 of \cite{ACGH}. (Moreover a stronger statement is valid.) It claims that if the number $\delta$ of nodes of an irreducible plane nodal curve $\Ga\subset \CP^2$ of degree $d$ satisfies the inequality 
$\delta<d-3$ then the linear system $g_d^2$ cut out on $\Ga$ by lines is complete and unique on the normalization $\C$ of $\Ga$.  This fact immediately implies that under the above assumptions two plane curves whose normalizations are isomorphic will be projectively equivalent. Then for degree at least $4$ it will be straight-forward that if the isomorphism of their normalizations is induced by the equivalence of the meromorphic functions obtained by projection from the same point $p$, then the projective transformation realizing this equivalence belongs to $G_3$, see the proof of Proposition~\ref{pr:smooth} below.  In general, one should show that for a generic curve in $S(d,l,g)$,   one has  $h^0(N,\OO_N(E)) =3$. This fact is also valid and will appear in a forth-coming publication \cite{ShTyo}.  
\end{proof}  

\begin{corollary}\label{cor:M(g,d)} Given $g,d$ as above, 
\begin{equation}\label{eq:M}
M(g,d)=\max{\left(0,\left\lceil \frac{g-d+2}{2}\right\rceil \right)}.
\end{equation}

\noindent
In particular, $m(g,d)=M(g,d)=0$ if and only if $d\ge g+2$.
\end{corollary}

\begin{proof} From Theorem~\ref{th:dim} is follows that $M(g,d)$ equals the minimal non-negative integer $l$ for which 
$$3d+g+2l-4\ge 2d+2g-2 \leftrightarrow 2l \ge g-d+2.$$ 
 The latter inequality implies that $M(g,d)=\max{\left(0,\left\lceil \frac{g-d+2}{2}\right\rceil \right)}.$ 
 This formula for $M(g,d)$ gives that $M(g,d)=0$ if and only if $d\ge g+2$. 
\end{proof}�

\begin{corollary}\label{cor:onepiece} The planarity stratification of $\HH_{g,d}$ consists of one term in the following two cases. 
Either $d\ge g+2$ in which case the planarity defect vanishes, or $d=3$ in which case the planarity defect equals $\lceil \frac{g-1}{2}\rceil$. 
\end{corollary}

\begin{proof} We have that $\HH_{g,d}$ consists of one term if and only if $m(g,d)=M(g,d)$. By Proposition~\ref{pr:filt} and Theorem~\ref{th:dim} (unless $M(g,d)$ vanishes which happens if and only if $d\ge g+2$) � this corresponds to the case when 
$$\left\lceil \frac{g-\binom{d-1}{2}}{d-1}\right\rceil=\left\lceil \frac{g-d+2}{2}\right\rceil.$$
If $d>3$ then the denominator of the left-hand side is smaller than that of the right-hand side and the numerator of the left-hand side is bigger than that of the right-hand side which means that the equality never holds. For $d=3$ the left-hand side  and the right-hand side coincide giving the planarity defect equal to $\lceil \frac{g-1}{2}\rceil$.
\end{proof}�

\subsection {Stratification of Hurwitz spaces with one complicated branching point} 
Analogously to the above, given a partition $\mu=(\mu_1\ge \mu_2\ge \dots \ge \mu_n)\vdash d$  of positive integer $d,$ denote by 

$$\HH_{g,\mu}=\{f: \C\to \CP^1\vert f\; \text{has all simple branched points except at } \infty$$
$$ \text{ whose profile is given by  $\mu$},  \;\deg f=d\ge 2,\; \text{gen}(\C)=g\ge 0\}$$
 the  Hurwitz space of all degree $d$ functions on genus $g$ curves with one complicated branching point at $\infty$ having  a given profile $\mu$. Recall that $\dim \HH_{g,\mu}$ equals the number of simple branching points of a function from $\HH_{g,\mu}$ and is given by the formula
$$\dim \HH_{g,\mu}=2d+2g-2-\sum_{i=1}^n(\mu_i-1). $$

Proposition~\ref{pr:birational} allows us to introduce the {\it planarity stratification} of $\HH_{g,\mu}$:
\begin{equation}\label{filtrmu}
 \HH_{g,\mu}^{m(g,\mu)}\subset \HH_{g,\mu}^{m(g,\mu)+1} \subset \dots \subset \HH_{g,\mu}^{M(g,\mu)}=\HH_{g,\mu}.
\end{equation} 
 Here $\HH^l_{g,\mu}$ consists of all meromorphic functions in $\HH_{g,\mu}$ whose defect does not exceed $l$. 

By Lemma~\ref{defec},  $M{(g,\mu)}\le d+2$. 

\begin{proposition}\label{pr:filtmu} For any pair $(g,\mu\vdash d)$ where $g\ge 0$ and $d\ge 2$, 

\begin{equation}\label{eq:BLA}
m{(g,\mu)}=\min_{l\ge 0}\binom{d+l-1}{2}-\binom{l}{2}\ge g. 
\end{equation}
which gives 
$$m{(g,\mu)}=\left\lceil \frac{g-\binom{d-1}{2}}{d-1}\right\rceil. $$
(Observe that $m(g,\mu)=m(g,d)$ given by \eqref{eq:mgd}.) 
\end{proposition}

\begin{proof}
 Since the stratum $\HH_{g,\,u}^{m(g,\mu)}$ should lie at least in $\HH_{g,d}^{m(g,d)}$ or, possibly in the higher strata of the planarity stratification of $\HH_{g,d}$. Therefore $m(g,\mu)$ is at least equal to the minimal $l$ given by the right-hand side of \eqref{eq:BLA}. The fact that $m(g,\mu)$ is exactly equal to the minimal $l$ satisfying the latter condition is explained in the proof of Theorem~\ref{th:dimmu}. 
\end{proof} 

 We have the following result above the dimensions of the strata of \eqref{filtrmu}. 

\begin{theorem}\label{th:dimmu} In the above notation, given $g,d$  and $l\ge m(g,\mu)$, the stratum $\HH_{g,\mu}^l$ is 
equidimensional  and its dimension is given by: 
\begin{equation}\label{eq:dimen}
\dim \HH_{g,\mu}^l=\min{(3d+g+2l-4-\sum_{i=1}^n(\mu_i-1), 2d+2g-2-\sum_{i=1}^n(\mu_i-1))}.  
\end{equation}
\end{theorem}

\begin{proof}
Theorem~\ref{th:dimmu} follows directly from Lemmas~\ref{lm:1} and \ref{lm:2}.
\end{proof}

Fix a flag $p\in L_0 \subset \CP^2$, positive integers $g,d,l$, and a partition $\mu\vdash d$. Consider the locus $V\subset |\OO_\CP^2(d+l)|$ of plane curves $\C$ such that: 
1) $\deg \C=d+l$;   2) $\C$ is reduced and irreducible;  3) $\text{mult}_p\C=l$;  4) $p_g(\C)=g$;  5) $\kappa^{-1}L_0=\sum_{i}\mu_iq_i$ where $\kappa: \tilde\C\to \C$ is the normalization map. 

Again let $\Sigma_1=Bl_p\CP^2$ be the first Hirzebruch surface obtained by the blow-up of $\CP^2$ at $p$. Let $F_0\subset \Sigma_1$ be the strict transform of $L_0$, and let $F$ be the class of $F_0$. Denote by $L\subset \Sigma_1$ the class of the preimage of a general line in $\CP^2$, and denote by $E\subset \Sigma_1$ the exceptional divisor.  Then $V$ can be identified with the locus of curves $\C\in |\OO_{\Sigma_1}((d+l)L-lE|=|\OO_{\Sigma_1}(dL+lF)|$ such that 
i) $\C$ is reduced and irreducible; ii) $p_g(\C)=g;$ iii) $\kappa^{-1}F_0=\sum_i\mu_iq_i$.  (Here $p_g(\C)$ is the geometric genus.) 

Let $V_1\subset V$ be an irreducible component of $V$. 

\begin{lemma}\label{lm:1}
$\dim V_1 \ge \rm{\text{exp}}\dim:= -K_{\Sigma_1}\cdot \C +g-1- \sum_{i=1}^n(\mu_i-1)$. 
\end{lemma}

\begin{proof}
Let $o\in V_1$ be a general point,  $\C_o$ be the corresponding curve. By \cite{KlSh} Lemma A.3 there exists a  neighborhood  $W$ of $o\in V_1$ over which the family $\C_W\to W$ is equinormalizable, i.e. if $\widetilde \C_W\to \C_W$ is the normalization then $\forall a\in W, (\widetilde \C_W)_a\to (\C_W)_a=\C_a$ is the normalization. Thus $\dim V_1$ is equal to the dimension of (a component of) the deformation space of $f:\widetilde \C_0\to \Sigma_1$ satisfying condition (iii). 
 Notice that condition (iii) has codimension $\le \sum_{i=1}^n(\mu_i-1)$ in the space of all deformations of the pair $(\tilde \C_0,f_0)$. Thus, it suffices to show that (any component of) $Def(\tilde \C_0,f_0)$ has dimension at least $-K_{\Sigma_1}\cdot \C+g-1$. By the standard deformation theory any component of the latter space has dimension $\ge \dim Def^\prime(\tilde \C_0,f_0)-\dim Ob(\tilde C_0,f_0)$. In our case $Def^\prime(\tilde \C_0,f_0)=H^0(\tilde \C_0, N_{f_0})$ and $Ob(\tilde \C_0,f_0)=H^1(\tilde \C_0, N_{f_0})$ where $N_{f_0}$ is the normal sheaf of $f_0$, i.e. $N_{f_0}=Coker (T_{\tilde \C_0}\to f_0^*T_{\Sigma_1})$. This implies the statement since $h^0(\tilde \C_0,N_{f_0})-h^1(\tilde C_0,N_{f_0})=\chi(\tilde C_0,N_{f_0})=-K_{\Sigma_1}\cdot \C+g-1$ by  Riemann-Roch's theorem.
\end{proof}

\begin{lemma}\label{lm:2}
$\dim V_1 \le \rm{\text{exp}}\dim V_1$. 
\end{lemma}

\begin{proof}
If $\dim V_1> \rm{\text{exp}}\dim$ then there exists a configuration of $r$ points on $F_0$ such that $\{\C\in V_1| \C\cap F_0=\text{given configuration}\}$ has dimension greater than $-K_{\Sigma_1}\cdot \C +g-1-\sum_{i=1}^n(\mu_i-1)-n=-K_{\Sigma_1}\cdot \C+g-1-F_0\cdot \C$, which is a contradiction with \cite{Tyo}, Lemma 2.9. 
\end{proof}

\begin{corollary}\label{cor:M(g,mu)} Given $g,\mu$ as above, 
\begin{equation}\label{eq:M}
M_{g,\mu}=\max{\left(0,\left\lceil \frac{g-d+2}{2}\right\rceil \right)}.
\end{equation}

\noindent
In particular, $m_{g,\mu}=M_{g,\mu}=0$ if and only if $d=\sum_{i=1}\mu_i\ge g+2$. 
\end{corollary}

\begin{proof} See the proof of Corollary~\ref{cor:M(g,d)}. 

\end{proof} 

Stratification~\eqref{filtr} is (almost) the special case of \eqref{filtrmu} the difference being that one simple branching point is placed at $\infty$.

\medskip \noindent
{\bf Remark.}  According to the information the author obtained from I.~Tyomkin one can prove that each stratum $\HH_{g,\mu}^l$ is irreducible for $g=0$ and $g=1$, and hopefully for other genera if $\mu \vdash d$ is  not very complicated. Whether  $\HH_{g,\mu}^l$ is irreducible for an arbitrary partition $\mu$
 is unknown at present and might be a difficult problem.

\section{ Hurwitz numbers of the planarity stratification and Zeuten-type problems} 

Due to irreducibility of strata of  \eqref{filtr} and equidimensionality of strata of \eqref{filtrmu}  we can introduce the correspoding notion of Hurwitz numbers related to these strata. Recall that the {\it branching morphism} 
\begin{equation}\label{eq:branch}
\delta_{g,d}: \HH_{g,d}\to \CP^{2d+2g-2}\setminus \Delta 
\end{equation}
is, by definition, the map sending a meromorphic function $f$ to the unordered set of its branching points (which are distinct by definiton). Here $\Delta\subset \CP^{2d+2g-2}$ is the hypersurface of unordered 
$(2d+2g-2)$-tuples of points in $\CP^1$ where not all of them are pairwise distinct. It is well-known that $\delta_{g,d}$ is a finite  covering and its degree $h_{g,d}$ is called the (small) Hurwitz number.  In particular, for $g=0$ the corresponding Hurwitz number $h_{0,d}$ equals $(2d-2)!d^{d-3}$.  In general, however closed formulas for $h_{g,d}$ (as well as for many other Hurwitz numbers) are unknown. 

Analogously, the  {\it branching morphism} 
\begin{equation}\label{eq:branch}
\delta_{g,\mu}: \HH_{g,\mu}\to \bC^{w_\mu}\setminus \Delta 
\end{equation}
is, by definition, the map sending a meromorphic function $f\in\HH_{g,\mu}$ to the unordered set of its simple branching points (which are distinct by definiton). Here $\Delta\subset \bC^{w_\mu}$ is the hypersurface of unordered 
$w_\mu$-tuples of points in $\bC$ where not all of them are pairwise distinct. Here $w_\mu=2d+2g-2-\sum_{i=1}^n(\mu_i-1)$. It is well-known that $\delta_{g,\mu}$ is a finite  covering and its degree $h_{g,\mu}$ is called the single Hurwitz number.  In particular, for $g=0$ the corresponding Hurwitz number $h_{0,\mu}$ equals 
$$(d+n-2)!\prod_{i=1}^n \frac{\mu_i^{\mu_i}}{\mu_i!}d^{n-3}.$$  

Stratifications~\eqref{filtr} -- \eqref{filtrmu} allow to introduce Hurwitz numbers which take into account these filtrations. Before we introduce this notion in general, let us start with a motivating example. 

\medskip
\noindent
{\bf Example.}
Fixing a point $p\in \CP^2$, consider the space $S_{d,p}$ of all smooth plane curves of degree $d$ not passing through $p$. Each such curve defines a branched covering of $\CP^1$ of degree $d$. 
There exists a three-dimensional group $G_p\subset PGL_3$ of projective transformations preserving $p$ as well as the pencil of lines through $p$. In other words, each line through $p$ will be mapped to itself. Obviously $G_p$ acts (locally) freely on $S_{d,p}$ for $d>1$ and curves from the same orbit define equivalent branched coverings of $\CP^1$. The following simple statement holds.

 As usual, two mappings
$p_1\colon C_1\to \CP^1$ and $p_2\colon C_2\to \CP^1$ are called
\emph{equivalent} if there exists an isomorphism $f\colon C_1\to C_2$
s.t.\ $p_2\circ f=p_1$.


\begin{proposition}\label{pr:smooth}
Suppose that $C_1,C_2\subset\CP^2$ are smooth projective curves of degree at least $4$ not
passing through $p$. Then the morphisms $\pi_{C_1}$ and $\pi_{C_2}$
are equivalent if and only if there exists an automorphism $f\in G_p$
s.t.\ $f(C_1)=C_2$.
\end{proposition}

\begin{proof}
The 'if' part being obvious, suppose that  $\pi_{C_1}$ and $\pi_{C_2}$
are equivalent and that this equivalence is performed by the
isomorphism $f\colon C_1\to C_2$. For each line $\ell\ni p$, the
isomorphism $f$ maps $C_1\cap \ell$ to $C_2\cap\ell$; thus, $f$ maps
hyperplane sections of $C_1$ to hyperplane sections of $C_2$. Since
both $C_1$ and $C_2$ are embedded in $\CP^2$ by the complete linear
system of plane sections, this implies that $F$ is induced by a
projective automorphism $F\in PGL_3$. It remains to check that
$F\in G_p$; to that end, consider a generic $\ell\ni p$; this line
intersects $C$ at $M=\deg C>1$ points $p_1,\dots,p_m$, and these
points are mapped by $F$ to $m$ distinct points on $\ell$. So,
$F(\ell)=\ell$ for the generic (whence for any) $\ell\ni p$. If
$\ell_1,\ell_2\ni p$, then
\[
F(p)=F(\ell_1\cap \ell_2)=F(\ell_1)\cap F(\ell_2)=\ell_1\cap \ell_2=p,
\]
which completes the proof.
\end{proof}

Denote by $\mathfrak h_d$ the number of different $3$-dimensional orbits of the above action on the space $S_{d,p}$ witth the same set of $d(d-1)$ tangent lines (e.g. branching points of the projection).

\noindent
{\bf Example.} One can easily observe that $\mathfrak h_2=2,\; \mathfrak h_3=40$ (which are the usual Hurwitz numbers for degree $d$ and genus $\binom {d-1}{2}$. But starting with $d=4$ the situation changes. So far the only calculated non-trivial example is $d=4$ see \cite{Va}, \cite {Va2}  for which $\mathfrak h_4=120\times (3^{10}-1)$. Numbers $\mathfrak h_d$ for $d>4$ are unknown at present. 

Observe a straight-forward analogy of the calculation of  $\mathfrak h_d$ with (a special case) of the classical Zeuten's problem, see \cite{Ze}, \cite{Al}. Namely, given $d\ge 2$ and  $0\le k\le \frac{d(d+3)}{2}$ define the number $N_k(d)$  as the number of smooth curves of degree $d$ passing through $k$ points in general position and tangent to  $\frac{d(d+3)}{2}-k$ lines in general position. In \cite{Ze} H.~G.~Zeuten predicted these numbers for $d$ up to $4$. His predictions were rigorously proven only in the 90's, see \cite{Al} and references therein.   The above problem of calculation of $\mathfrak h_d$ is similar to Zeuten's problem for $k=\frac{d(d+3)}{2}$. But instead of taking $\frac{d(d+3)}{2}$ generic lines we should take $\frac{d(d+3)}{2}-3$ generic lines through a given point $p$ and count the number of $3$-dimensional orbits under the action of $G_p$. 

\begin{definition} Introduce the Hurwitz number $\mathfrak h_{g,\mu}^l$ as the degree of the restriction of the morphism $\delta_{g,\mu}$ to the (irreducible component of the) stratum 
$\HH_{g,\mu}^l$ where $m(g,\mu)\le l\le M(g,\mu)$. 
\end{definition} 

By definition, $\mathfrak h_{g,\mu}^{M(g,\mu)}=h_{g,\mu}$. Also the number $\mathfrak h_d$ introduced above equals $\mathfrak h_{(d-1)(d-2)/2, 1^d}^0$.

\section {Final Remarks}

\medskip\noindent 
{\bf 1.} It would be very interesting to prove/disprove the irreducibility of the strata $\HH_{g,\mu}^l$.

\medskip\noindent 
{\bf 2.} It is important to develop  tools helping for calculation of  the Hurwitz numbers of $\HH_{g,d}^l$ and/or $\HH_{g,\mu}^l$ due to the fact that they are  naturally related  to Zeuten-type problems.  In the case of the usual single Hurwitz numbers there exists a standard combinatorial approach to the calculation of those which is not always very useful for practical computations but is very important theoretically. Other standard tools for the usual Hurwitz numbers are the cut-and-join equation, see e.g. \cite{GJV} and the ELSV-formula, see e.g. \cite{ELSV}. It might be possible to find analogs of the latter tools by using an appropriate compactification of the above strata similar to those already existing in the literature. 

\medskip\noindent 
{\bf 3.} Another approach to the calculation of the Hurwitz strata of the planarity filtration might come from the correspondence theorem in tropical algebraic geometry. Recently in \cite {BBM} the authors developed some tropical tools for finding the answers to a similar class of Zeuten-type problems. 

\medskip\noindent 
{\bf 4.} Finally, we want to mention a recent preprint  \cite{BuLv} which gives a criterion when meromorphic functions of degree $d$ on a certain class of plane curves of degree $d$ with only  nodes and some additional non-degeneracy assumptions might be realized by a projection from a point outside the curve.

\medskip 
\noindent 
{\bf Acknowledgements.}  We want to thank   O.~Bergvall,  S.~Shadrin for discussions, and especially  I.~Tyomkin for his explanations of \cite{Tyo} and his help with the proofs of Theorem~\ref{th:dim} and Theorem~\ref{th:dimmu}.

\end{document}